\documentclass[12pt,a4paper,reqno]{amsart}

\usepackage[utf8]{inputenc}

\usepackage{amssymb}
\usepackage{amsfonts}
\usepackage{amsmath}
\usepackage[mathscr]{eucal}

\usepackage{color}

\newtheorem{thm}{Theorem}[section]
\newtheorem{prop}[thm]{Proposition}

\newtheorem{rem}[thm]{Remark}

\numberwithin{equation}{section}
\numberwithin{equation}{section}

\def\GL{{\operatorname {GL}}}

\def\SO{{\operatorname{SO}}}
\def\U{{\operatorname {U}}}

\def\Re{{\operatorname {Re}}}

\def\leq{\leqslant}
\def\geq{\geqslant}

\def\le{\leq}
\def\ge{\geq}

\def\1{{\bold 1}}

\renewcommand{\a}{\alpha}

\newcommand{\s}{\sigma}

\newcommand{\Ecal}{{\mathcal E}}
\newcommand{\Fcal}{{\mathcal F}}

\newcommand{\Ocal}{{\mathcal O}}

\newcommand{\CC}{\mathbb{C}}

\newcommand{\QQ}{\mathbb{Q}}
\newcommand{\RR}{\mathbb{R}}

\newcommand{\Sp}{\operatorname{Sp}}

\renewcommand{\Re}{\operatorname{Re}}

\newcommand{\supp}{\operatorname{supp}}

\title[Weighted one-level density of low-lying zeros of $L(s,\chi)$]{Weighted one-level density of low-lying zeros of Dirichlet $L$-functions}

\author[Sugiyama]{Shingo Sugiyama}
\address{Department of Mathematics, College of Science and Technology, Nihon University, Suruga-Dai, Kanda, Chiyoda, Tokyo 101-8308, Japan}
\email{sugiyama.shingo@nihon-u.ac.jp}

\author[Suriajaya]{Ade Irma Suriajaya}
\address{Faculty of Mathematics, Kyushu University, 744 Motooka, Nishi-ku, Fukuoka 819-0395, Japan}
\email{adeirmasuriajaya@math.kyushu-u.ac.jp}



\begin{document}

\begin{abstract}
In this paper, we compute the one-level density
of low-lying zeros of Dirichlet $L$-functions in a family weighted by special values of Dirichlet $L$-functions at a fixed $s \in [1/2, 1)$.
We verify both Fazzari's conjecture
and the first author's conjecture on the weighted one-level density for our family of $L$-functions.
\end{abstract}

\maketitle

\section{Introduction}

The Riemann zeta function $\zeta(s)$ is among the most interesting objects in number theory for its close relation to the distribution of prime numbers, which are fundamental objects in number theory. In particular, zeros of $\zeta(s)$ detect prime numbers and the study of zeros of $\zeta(s)$ is inevitable in order to understand the distribution of prime numbers.
For example, the famous Riemann hypothesis, which asserts that $\Re(\rho)=1/2$ should hold for all zeros $\rho$ of $\zeta(s)$ lying in the critical strip $0<\Re(s)<1$, called non-trivial zeros of $\zeta(s)$, is among the most important unsolved problems in mathematics.
One interpretation of the Riemann hypothesis links non-trivial zeros of $\zeta(s)$ to eigenvalues of a certain operator.
More precisely, it is said that the Riemann hypothesis is implied by the
Hilbert-P\'olya conjecture, 
which asserts the existence of a determinant expression of $\zeta(s)$
using a Hamiltonian $H$ in the form of $\zeta(1/2+it)=``\det(t\, {\rm id}-H)"$.
In 1973, Montgomery \cite{Montgomery} gave evidence of the Hilbert-P\'olya conjecture, and later Odlyzko \cite{Odlyzko} gave supporting numerical data.
The so-called Montgomery-Odlyzko law predicts
that non-trivial zeros of $\zeta(s)$ are
distributed like eigenvalues of random Hermitian matrices in the Gaussian Unitary Ensemble.
At the end of the twentieth century,
Katz and Sarnak \cite{KS1,KS2} shed light on a family of $L$-functions instead of an individual $L$-function in order to relate zeros of $L$-functions
to eigenvalues of random matrices.
The Katz-Sarnak philosophy (or called the density conjecture)
predicts that the distribution of low-lying zeros of $L$-functions in a family
is similar to that of the eigenvalues of random matrices, and that the family of $L$-functions has one of five symmetry types
$\U$, $\Sp$, $\SO({\rm even})$, $\SO({\rm odd})$ and $\rm O$
(unitary, symplectic, even orthogonal, odd orthogonal and orthogonal) in accordance with the density of low-lying zeros of $L$-functions in the family.
Shortly thereafter, Iwaniec, Luo and Sarnak \cite{ILS}
confirmed the density conjecture
for the case of automorphic $L$-functions
attached to elliptic modular forms and their symmetric square lifts.

Motivated by the work of Knightly and Reno \cite{KnightlyReno} in 2019,
the one-level density for a family of $L$-functions {\it weighted by $L$-values} has been studied
in the context of random matrix theory.
Knightly and Reno \cite{KnightlyReno} discovered the phenomenon that
the symmetry type of a family of automorphic $L$-functions attached to elliptic modular forms changes {\it from orthogonal to symplectic} due to the weight factors of central $L$-values, by comparing
\cite{KnightlyReno} with the usual one-level density \cite{ILS}.
Their work was inspired by Kowalski, Saha and Tsimerman \cite{KST},
who treated the one-level density for a family of spinor $L$-functions attached to Siegel modular forms of degree $2$ weighted by Bessel periods, identical to central $L$-values by \cite{FurusawaMorimoto}.
We notice by comparing the weighted one-level density \cite{KST} and the usual one-level density \cite{KWY1,KWY2} that the symmetry type of the family of those spinor $L$-functions changes {\it from orthogonal to symplectic}.
Recently, the first author \cite{Sugiyama} observed
a new phenomenon occurring in a family of symmetric square $L$-functions attached to Hilbert modular forms.
He found that its symmetry type changes {\it from symplectic to
a new type} of density function which does not occur in random matrix theory as Katz and Sarnak predicted.
Afterwards,
Fazzari \cite{Fazzari} conjectured
that for a family of $L$-functions whose symmetry type is unitary, symplectic or even orthogonal,
the one-level density for the family weighted by central $L$-values should coincide with the density of
eigenvalues of random matrices
weighted by their characteristic polynomials.
He gave evidence of his conjecture
by studying the following three families of $L$-functions under the generalized Riemann hypothesis and the Ratios Conjecture:
$\{\zeta(s+ia)\}_{a \in \RR}$,
$\{L(s,\chi_d)\}_{d}$
and $\{L(s, \Delta\otimes \chi_d)\}_{d}$.
Here, $d$ is a fundamental discriminant, $\chi_d$ is the primitive quadratic Dirichlet character corresponding to
the quadratic field $\QQ(\sqrt{d})$ by class field theory,
and $\Delta$ is the delta function (a cusp form of weight $12$ and level $1$).
The first author's result \cite{Sugiyama} on the change of the symmetry type provides new evidence of Fazzari's conjecture.

In this paper, we consider the one-level density for a family of Dirichlet $L$-functions weighted by special values of Dirichlet $L$-functions at $s \in [1/2, 1)$.
We are interested to see whether the symmetry type of the family of Dirichlet $L$-functions under consideration changes.
For a prime number $q$, let $\Fcal_q$ be the set of all non-principal Dirichlet characters modulo $q$.
We consider the one-level density of low-lying zeros of
the non-completed Dirichlet $L$-function $L(s, \chi)$ attached to a Dirichlet character $\chi \in \Fcal_q$ defined as
\[D(\chi, \phi):= \sum_{\rho=1/2+i\gamma}\phi\left(\frac{\log q}{2\pi}\gamma\right),\]
where a test function $\phi$ is a Schwartz function on $\RR$
such that the support $\supp(\hat\phi)$ of the Fourier transform
$\hat\phi(\xi):=\int_{-\infty}^{\infty}\phi(x)e^{-2\pi i \xi x}dx$
of $\phi$ is compact, 
and $\rho$ runs over the set of all zeros of $L(s, \chi)$
in the critical strip $0<\Re(s)<1$ counted with multiplicity.
We remark that $\phi$ is naturally extended to a Paley-Wiener function on $\CC$. With this extension, the assumption of the generalized Riemann hypothesis is not necessary and we emphasize that we do not assume the generalized Riemann hypothesis throughout this paper. Thus $\gamma$ in the summation of $D(\chi, \phi)$ is not necessarily real.

We consider the weighted one-level density
of low-lying zeros of $L(s, \chi)$ for $\chi \in \Fcal_q$.
The usual one-level density for $\Fcal_q$ has been given by Hughes and Rudnick \cite[Theorem 3.1]{HughesRudnick} as follows.
\begin{thm} [Hughes and Rudnick, 2003]
For a Schwartz function $\phi$ on $\RR$ such that
$\supp(\hat\phi) \subset [-2, 2]$, we have
\[\frac{1}{\#\Fcal_q}\sum_{\chi\in \Fcal_q}D(\chi, \phi)
= \int_{-\infty}^{\infty}\phi(x)W_{\rm U}(x)dx +\Ocal\left(\frac{1}{\log q}\right), \quad q\rightarrow \infty\]
with
\[W_{\rm U}(x)=1,\]
where $q$ tends to infinity in the set of prime numbers.
Thus, the symmetry type of the family $\bigcup_{q}\{L(s, \chi)\mid \chi \in \Fcal_q\}$
is unitary.
\end{thm}

We consider the one-level density for $\Fcal_q$ weighted by the square of the absolute values of central $L$-values.
We note that, by the asymptotic behavior of the second moment
\[\sum_{\chi \in \Fcal_{q}}|L(1/2, \chi)|^2 \sim \frac{(q-1)^2}{q}\log q, \qquad q \rightarrow \infty\]
of Dirichlet $L$-functions
due to Paley \cite[Theorem II]{Paley},
the second moment of Dirichlet $L$-functions is non-zero for any sufficiently large prime number $q$.
Our weighted one-level density result is stated as follows.
\begin{thm}\label{main by L(1/2)}
Let $\phi$ be a Schwartz function on $\RR$
such that
$\supp(\hat\phi) \subset [-1/3, 1/3]$.
Then, we have
\[\frac{1}{\sum_{\chi \in \Fcal_{q}}|L(1/2, \chi)|^2}\sum_{\chi \in \Fcal_q}|L(1/2, \chi)|^2 D(\chi, \phi)=\int_{-\infty}^{\infty}\phi(x)W_{\rm U}^{1}(x)dx
+\Ocal\left(\frac{1}{\log q}\right), \quad q \rightarrow \infty\]
with
\[W_{\rm U}^1(x)=1-\frac{\sin^2(\pi x)}{(\pi x)^2},\]
where $q$ tends to infinity in the set of prime numbers.
Thus when the central $L$-values are weight factors, the density function $W_{\U}$ of the one-level density changes to $W_{\U}^1$, which coincides with the pair correlation function of random Hermitian matrices in the Gaussian Unitary Ensemble.
\end{thm}

\begin{rem}
Theorem \ref{main by L(1/2)} gives another concrete example of the weighted one-level density conjecture by Fazzari \cite[Conjecture 1]{Fazzari}.
\end{rem}

Next we consider the one-level density weighted by the square $|L(s,\chi)|^2$ of special $L$-values at
any fixed $s \in (1/2,1)$.
By \eqref{twisted second moment} in Proposition \ref{Selberg by L(s)} below, we have
\[\sum_{\chi \in \Fcal_q}|L(s, \chi)|^2
\sim q\,\zeta(2s), \qquad q\rightarrow \infty.
\]
Hence this sum is non-zero
for any sufficiently large prime number $q$.
Our weighted one-level density with parameter $s$
is stated as follows.
\begin{thm}\label{main by L(s)}
Take any $s\in (1/2, 1)$.
Let $\phi$ be a Schwartz function on $\RR$ such that 
$\supp(\hat\phi) \subset[-2s/3, 2s/3]$.
Then we have
\begin{align*}
\frac{1}{\sum_{\chi \in \Fcal_q}|L(s, \chi)|^2}\sum_{\chi \in \Fcal_q}
|L(s, \chi)|^2D(\chi, \phi) =\int_{-\infty}^{\infty}\phi(x)W_{\rm U}(x)dx
+\Ocal_s\left(\frac{1}{\log q}\right).
\end{align*}
Thus the change of symmetry type does not occur
when $1/2< s< 1$.
\end{thm}

\begin{rem}
By Theorems \ref{main by L(1/2)} and \ref{main by L(s)},
the weighted one-level density by special values $|L(s,\chi)|^2$ at $s\in[1/2, 1)$
causes the change of symmetry type
if and only if $s=1/2$.
This supports the weighted density conjecture by the first author \cite[Conjecture 1.3]{Sugiyama}.
\end{rem}

The key idea for proving
the asymptotic behavior of the weighted one-level density is the use of Selberg's formula of the twisted second moment of Dirichlet $L$-functions with complex parameters $s$ and $s'$ \cite{Selberg}.
Selberg's formula is a substitute
of the explicit Jacquet-Zagier type trace formula
by the first author and Tsuzuki \cite{SugiyamaTsuzuki},
as we see that such a parametrized trace formula was used in \cite{Sugiyama}
for analysis of the weighted one-level density for symmetric power $L$-functions attached to Hilbert modular forms.
The computation in Proposition \ref{Selberg by L(1/2)}
is essentially the same as that in \cite[Theorem 2.6]{Sugiyama}.
Both computations include the derivation at $s=1/2$
for deducing the main term and the second main term.
Asymptotics using Selberg's formula is explained in \S \ref{Twisted second moment of Dirichlet L-functions}.
We give proofs of Theorems \ref{main by L(1/2)} and \ref{main by L(s)}
in \S \ref{Weighted one-level density}.

\begin{rem}
We can extend Theorems \ref{main by L(1/2)} and \ref{main by L(s)} to the case of general $q$, that is, $q$ does not necessarily need to be a prime number. In that case, in place of Theorem \ref{Selberg original}, we need to use the more general form of Selberg's formula as stated in \cite[pp. 6--7]{Selberg}.
\end{rem}

\section{Twisted second moment of Dirichlet $L$-functions}
\label{Twisted second moment of Dirichlet L-functions}

In this section, we use the asymptotic formula for the twisted second moment of Dirichlet $L$-functions due to Selberg \cite[Theorem 1]{Selberg} to deduce the asymptotic formula suitable for our purpose.
Recall that we consider only the case when the modulus $q$ of Dirichlet characters is a prime number.

\begin{thm}[Selberg, 1946] \label{Selberg original}
Let $q$ be a prime number.
Let $m$ and $n$ be positive integers such that
$m$ and $n$ are coprime to each other and to $q$.
For any $(s ,s')\in \CC^2$ such that $\s=\Re(s)\in (0, 1)$ and $\s'=\Re(s')\in(0, 1)$, we have
\begin{align*}
& \sum_{\chi\in {\Fcal_q}}L(s, \chi)L(s', \overline{\chi})
\chi(m)\overline{\chi(n)} \\
=&\, \frac{q-1}{m^{s'}n^{s}}\zeta(s+s') \\
&
+\frac{(q-1)\, q^{1-s-s'}}{m^{1-s}n^{1-s'}}\frac{(2\pi)^{s+s'-1}}{\pi}\Gamma(1-s)\Gamma(1-s')\cos\left(\frac{\pi}{2}(s-s')\right)\zeta(2-s-s')\\
&
+\Ocal\left(\frac{|ss'|}{\s\s'(1-\s)(1-\s')}(mq^{1-\s}+nq^{1-\s'}+mnq^{1-\s-\s'})\right), \qquad q\rightarrow \infty,
\end{align*}
where the implied constant is independent of $m$, $n$, $q$, $s$ and $s'$.
On the right-hand side, the value at $(s, s')$ with $s+s'=1$ is
understood as the limit when $s+s' \rightarrow 1$.
\end{thm}

\begin{prop}\label{Selberg by L(s)}
Let $q$ be a prime number and $m$ a positive integer coprime to $q$.
For any fixed $s \in (1/2,1)$, we have
\begin{equation} \label{prop moment L(s)}
\frac{1}{\sum_{\chi \in \Fcal_q}|L(s, \chi)|^2}\sum_{\chi\in {\Fcal_q}}|L(s, \chi)|^2\chi(m)
= m^{-s}+\Ocal_s(m^{s-1}q^{1-2s})+\Ocal_{s}(m q^{-s}), \quad q \rightarrow \infty,
\end{equation}
where the implied constant is independent of $m$ and $q$.
\end{prop}
\begin{proof}
Restricting Selberg's formula in Theorem \ref{Selberg original} to the case $s=s' \in (1/2,1)$ and $n=1$,
we obtain
\begin{align}\label{twisted second moment}
& \sum_{\chi\in {\Fcal_q}}|L(s, \chi)|^2\chi(m) \\
=&\, \frac{q-1}{m^s}\zeta(2s)
+ \frac{(q-1) q^{1-2s}}{m^{1-s}}\frac{(2\pi)^{2s-1}}{\pi}
\Gamma(1-s)^2\zeta(2-2s)
+ \Ocal_{s}(m q^{1-s}).\notag
\end{align}
Let $F_m$ denote
the main term of the right-hand side above and $E_m$
its error term.
Then the left-hand side of \eqref{prop moment L(s)} is equal to
\begin{equation} \label{FmEm}
\frac{F_m+E_m}{F_1+E_1}=\frac{F_m}{F_1}+\frac{E_m-\frac{F_m}{F_1}E_1}{F_1+E_1}
\end{equation}
(see \cite[Proposition 3.1]{KnightlyReno} and \cite[Corollary 2.9]{Sugiyama}).
The first term on the right-hand side
is then evaluated as
\[\frac{F_m}{F_1} = \frac{\frac{1}{m^s}
+\frac{q^{1-2s}}{m^{1-s}}\frac{(2\pi)^{2s-1}}{\pi}\Gamma(1-s)^2\frac{\zeta(2-2s)}{\zeta(2s)}}{
1+q^{1-2s}\frac{(2\pi)^{2s-1}}{\pi}\Gamma(1-s)^2\frac{\zeta(2-2s)}{\zeta(2s)}}
= m^{-s}
+\Ocal_s\left(\frac{1}{m^{1-s}q^{2s-1}}\right).\]
Finally, we estimate the second term of the right-hand side of \eqref{FmEm}
as
\begin{align*}
\frac{E_m-\frac{F_m}{F_1}E_1}{F_1+E_1}
\ll_{s} \frac{mq^{1-s}+(m^{-s}+m^{s-1}q^{1-2s})q^{1-s}}
{q}
\ll mq^{-s}.
\end{align*}
This completes the proof.
\end{proof}

\begin{prop}\label{Selberg by L(1/2)}
Let $q$ be a prime number and $m$ a positive integer coprime to $q$.
Then we have
\begin{align*}
& \frac{1}{\sum_{\chi \in \Fcal_q}|L(1/2, \chi)|^2}\sum_{\chi\in {\Fcal_q}}|L(1/2, \chi)|^2 \chi(m) \\
= &\, m^{-1/2}-m^{-1/2}\frac{\log m}{\log q}
+\Ocal\left(\frac{m^{-1/2}}{\log q}\left(1+\frac{\log{m}}{\log{q}}\right)\right)+\Ocal\left(\frac{mq^{-1/2}}{\log{q}}\right), \qquad 
q \rightarrow \infty,
\end{align*}
where the implied constant is independent of $m$ and $q$.
\end{prop}
\begin{proof}
When $s'=1/2$ and $n=1$, the main term of Selberg's formula in Theorem \ref{Selberg original} is given by
\begin{align}
\label{main term Selberg 1/2}
\frac{q-1}{m^{1/2}}\zeta(s+1/2)
+ (q-1)q^{1/2-s}m^{s-1}\frac{(2\pi)^{s-1/2}}{\sqrt{\pi}}\Gamma(1-s)\cos\left(\frac{\pi}{2}(s-1/2)\right) \zeta(3/2-s).
\end{align}
Utilizing
\[\zeta(s+1/2)=\frac{1}{s-1/2}+\gamma+\Ocal(s-1/2), \quad s\rightarrow 1/2,\]
\[\zeta(3/2-s)=\frac{-1}{s-1/2}+\gamma+\Ocal(s-1/2),
\quad s\rightarrow 1/2,\]
where $\gamma$ is the Euler-Mascheroni constant,
and
$\cos(\frac{\pi}{2}(s-1/2))=1+\Ocal((s-1/2)^2)$ as
$s\rightarrow 1/2$,
we can rewrite the term \eqref{main term Selberg 1/2} as
\begin{align}
& \left\{(q-1)m^{-1/2} - (q-1)q^{1/2-s}m^{s-1}
\frac{(2\pi)^{s-1/2}}{\sqrt{\pi}}\Gamma(1-s)\right\}\frac{1}{s-1/2}\label{first term Selberg 1/2}\\
& + (q-1)m^{-1/2}\gamma + (q-1)q^{1/2-s}m^{s-1}\frac{(2\pi)^{s-1/2}}{\sqrt{\pi}}\Gamma(1-s)\gamma \label{second term Selberg 1/2} \\
&+ \Ocal_{q,m}\left((s-1/2)\right),
\qquad s\rightarrow 1/2.
\notag
\end{align}
Applying L'H\^ospital's rule,
we see that the term \eqref{first term Selberg 1/2}
tends to
\begin{align*}
m^{-1/2}(q-1)\left(\log q - \log m - \gamma - \log{8\pi}\right)
\end{align*}
as $s\rightarrow 1/2$, where we use $\frac{\Gamma'(1/2)}{\Gamma(1/2)}=-\gamma-2\log 2$. Summing this and the value 
$2m^{-1/2}(q-1) \gamma$ of \eqref{second term Selberg 1/2} at $s=1/2$, and recalling the error term in Theorem \ref{Selberg original}, we obtain
\begin{align*}
\sum_{\chi \in \Fcal_q}|L(1/2,\chi)|^2\chi(m)
= \frac{q-1}{m^{1/2}}(\log q-\log m+\gamma-\log 8\pi)+\Ocal(m q^{1/2}).
\end{align*}
Let $F_m$ denote the main term of the right-hand side above 
and $E_m$ its error term.
As \eqref{FmEm}, we have
\[\frac{1}{\sum_{\chi \in \Fcal_q}|L(1/2, \chi)|^2}\sum_{\chi\in {\Fcal_q}}|L(1/2, \chi)|^2 \chi(m) = \frac{F_m+E_m}{F_1+E_1}=
\frac{F_m}{F_1}+\frac{E_m-\frac{F_m}{F_1}E_1}{F_1+E_1},\]
where $\frac{F_m}{F_1}$ equals
\begin{align*} m^{-1/2}\frac{\log q - \log m + \gamma - \log 8\pi}{\log q + \gamma - \log 8\pi}
&= m^{-1/2}\frac{1 -\frac{\log m}{\log q} +\frac{\gamma-\log{8\pi}}{\log q}}
{1 +\frac{\gamma -\log 8\pi}{\log q}} \\
&= \, m^{-1/2}\left(1-\frac{\log m}{\log q}\right)
+\Ocal\left(\frac{m^{-1/2}}{\log q} \left(1+\frac{\log{m}}{\log{q}}\right)\right).
\end{align*}
Furthermore, we have
\begin{align*}
\frac{E_m-\frac{F_m}{F_1}E_1}{F_1+E_1}
\ll \frac{mq^{1/2}+m^{-1/2}(1+\frac{\log m}{\log q})q^{1/2}}{q\log q}
\ll \frac{m}{q^{1/2}\log{q}}
\end{align*}
and the proof is complete.
\end{proof}

\section{Weighted one-level density}
\label{Weighted one-level density}
In this section, we prove Theorems \ref{main by L(1/2)} and \ref{main by L(s)}.
Let $\phi$ be a Paley-Wiener function on $\CC$.
Then for any $\chi \in \Fcal_q$ with  $q$ being a prime number,
the explicit formula for $L(s, \chi)$ \`a la Weil (see \cite[(2.2)]{HughesRudnick})
yields
\begin{align*}
D(\chi,\phi) = & \,\hat\phi(0) -\frac{1}{\log q} \sum_{p}
(\chi(p)+\overline{\chi(p)})
\hat\phi\left(\frac{\log p}{\log q}\right)\frac{\log p}{p^{1/2}}\\
& -\frac{1}{\log q} \sum_{p}
(\chi(p)^2+\overline{\chi(p)}^2)
\hat\phi\left(\frac{2\log p}{\log q}\right)\frac{\log p}{p}
+\Ocal\left(\frac{1}{\log q}\right), \qquad q\rightarrow \infty,
\end{align*}
where $p$ runs over the set of prime numbers.
For any $s \in [1/2, 1)$, set
\[\Ecal_{q}(s; \phi) := \frac{1}{\sum_{\chi \in \Fcal_q}|L(s, \chi)|^2}\sum_{\chi \in \Fcal_q}
|L(s, \chi)|^2 D(\chi, \phi).\]

\begin{prop}\label{asymptotics L(s)}
Fix $s \in (1/2, 1)$.
Let $\phi$ be a Schwartz function on $\RR$ such that $\supp (\hat\phi) \subset [-2s/3, 2s/3]$. Then we have
	\begin{align*}
	\Ecal_{q}(s; \phi)
	= \hat\phi(0)+\Ocal_s\left(\frac{1}{\log q}\right), \qquad q \rightarrow \infty.
	\end{align*}
\end{prop}
\begin{proof}Suppose $\supp(\hat\phi)\subset [-\a,\a]$, where $\a>0$
is suitably chosen later.
We recall the expression
\[\Ecal_q(s; \phi) = \hat\phi(0) -M_q^{(1)}(s)-M_q^{(2)}(s)+\Ocal\left(\frac{1}{\log q}\right),\]
where we put
\[M_q^{(k)}(s) := \frac{1}{\log q}\sum_{p}
\frac{1}{\sum_{\chi \in \Fcal_q}|L(s, \chi)|^2}\sum_{\chi \in \Fcal_q}|L(s, \chi)|^2(\chi(p^k)+\overline{\chi(p^k)})
\hat\phi\left(\frac{k\log p}{\log q}\right)\frac{\log p}{p^{k/2}}\]
for $k=1,2$. By Proposition \ref{Selberg by L(s)}, we have
\[M_q^{(k)}(s) = \sum_{p}2\left(\frac{1}{p^{ks}}+\Ocal_s(p^{k(s-1)}q^{1-2s})+\Ocal_{s}(p^kq^{-s})\right)
\hat\phi\left(\frac{k\log p}{\log q}\right)\frac{\log p}{p^{k/2}\log q},
\]
where we use $\sum_{\chi \in \Fcal_q}|L(s,\chi)|^2 \overline{\chi(p^k)}=\sum_{\chi \in \Fcal_q}|L(s, \chi)|^2\chi(p^k)$.
Since we assume $s>1/2$, the contribution of the term with $\frac{1}{p^{ks}}$ is
\begin{align*}\frac{2}{\log q}\sum_{p}\hat\phi\left(\frac{k\log p}{\log q}\right)\frac{\log p}{p^{k(s+1/2)}}\ll \frac{1}{\log q}\sum_p\frac{\log p}{p^{s+1/2}}\ll \frac{1}{\log q}.
\end{align*}
Furthermore,
noting $\supp(\hat \phi) \subset [-\a, \a]$,
the term $\Ocal_{s}(p^kq^{-s})$ is estimated as
\[\frac{2}{\log q}\sum_{p}p^kq^{-s}
\hat\phi\left(\frac{k\log p}{\log q}\right)\frac{\log p}{p^{k/2}}
\ll \frac{q^{-s}}{\log q}\sum_{p\le q^{\a/k}}p^{k/2}\log p.\]
Here, by the prime number theorem and partial summation, we have
\begin{align}\label{esti p log p}\sum_{p\le x}p^a\log p=\Ocal(x^{a+1}), \qquad x \rightarrow \infty\end{align}
for any fixed $a>-1$ (see \cite[(3.4)]{Sugiyama}).
As a consequence, the contribution of $\Ocal_{s}(p^kq^{-s})$ is bounded by
\[\frac{q^{-s}}{\log q}
\times (q^{\a/k})^{k/2+1}\le\frac{q^{-s+3\a/2}}{\log q}\]
up to constant multiple. This is absorbed into $\Ocal(\frac{1}{\log q})$
when $\a$ is taken so that $-s+3\a/2 \le 0$,
i.e., $\a\le 2s/3$.
Similarly,
the contribution of $\Ocal_s(p^{k(s-1)}q^{1-2s})$ is
bounded by
\[\frac{1}{\log q}\sum_{p}p^{k(s-3/2)}q^{1-2s}
\hat\phi\left(\frac{k\log p}{\log q}\right)\log p
\ll \frac{q^{1-2s}}{\log q}\sum_{p\le q^{\a/k}}p^{k(s-3/2)}\log p
\ll \frac{q^{1-2s}}{\log q}\times q^{\a(s-1/2)},\]
which is again absorbed into $\Ocal(\frac{1}{\log q})$
as long as $1-2s+\a(s-1/2)\le 0$, i.e., $\a \le 2$.
Hence $M_q^{(k)}(s)$ for $k=1, 2$ are both bounded by the error and the proof is done.
\end{proof}

Theorem \ref{main by L(s)} immediately follows from
Proposition \ref{asymptotics L(s)}.

\medskip

Next we consider $\Ecal_q(1/2; \phi)$.

\begin{prop}\label{asymptotics L(1/2)}
Let $\phi$ be a Schwartz function on $\RR$ such that $\supp (\hat\phi) \subset [-1/3,1/3]$.
Then we have
\begin{align*}
\Ecal_{q}(1/2; \phi)
	= \hat\phi(0)-\phi(0)+\int_{-\infty}^{\infty}\hat\phi(x)|x|dx +\Ocal\left(\frac{1}{\log q}\right), \qquad q\rightarrow \infty.
	\end{align*}
\end{prop}
\begin{proof}
Suppose $\supp(\hat\phi)\subset [-\a,\a]$, where $\a>0$
is suitably chosen later.
We write $\Ecal_{q}(1/2; \phi)$ as
\[\Ecal_{q}(1/2; \phi)
=\hat\phi(0)
+\Ocal\left(\frac{1}{\log q}\right)-M_q^{(1)}
-M_q^{(2)},\]
where we set
\[M_q^{(k)} :=\frac{1}{\log q}\sum_{p}
\frac{1}{\sum_{\chi \in \Fcal_q}|L(1/2, \chi)|^2}\sum_{\chi \in \Fcal_q}
|L(1/2, \chi)|^2(\chi(p^k) +\overline{\chi(p^k)})\hat\phi\left(\frac{k\log p}{\log q}\right)
\frac{\log p}{p^{k/2}}\]
for $k=1, 2$.
As in the proof of Proposition \ref{asymptotics L(s)}, with the aid of Proposition \ref{Selberg by L(1/2)}, the sum $M_q^{(k)}$ is evaluated as
\begin{align}M_q^{(k)} =\, & \sum_{p}\frac{2}{p^{k/2}}\left(1-\frac{k\log p}{\log q}
+\Ocal\left(\frac{1}{\log q}\left(1+\frac{\log{p}}{\log{q}}\right)\right)
\right)\hat\phi\left(\frac{k\log p}{\log q}\right)\frac{\log p}{p^{k/2}\log q}
\label{main of M}\\
& +\Ocal\left(\frac{q^{-1/2}}{(\log q)^2}\sum_{p}p^{k/2}\hat\phi\left(\frac{k\log p}{\log q}\right)\log p \right). \label{error of M}
\end{align}
Invoking the two asymptotic formulas
\begin{align*}
\sum_{p}\hat\phi\left(\frac{\log p}{\log q}\right)\frac{\log p}{p\log q}
=&\, \frac{1}{2}\phi(0) +\Ocal\left(\frac{1}{\log q}\right), \qquad q\rightarrow \infty,
\end{align*}
\begin{align*}
\sum_{p}\hat\phi\left(\frac{\log p}{\log q}\right)\frac{(\log p)^2}{p(\log q)^2}
=& \,\frac{1}{2}\int_{-\infty}^{\infty}\hat\phi(x)|x|dx+\Ocal\left(\frac{1}{\log q}\right), \qquad q\rightarrow \infty
\end{align*}
(see the proof of \cite[Proposition 3.2]{Sugiyama}),
the term \eqref{main of M} for $k=1$ is evaluated as
\[\phi(0)-\int_{-\infty}^{\infty}\hat\phi(x)|x|dx+\Ocal\left(\frac{1}{\log q}\right).\]
The term \eqref{main of M} for $k=2$ is bounded by
$$\frac{1}{\log q}\sum_{p}\frac{(\log p)^2}{p^2}\ll \frac{1}{\log q}.$$
The error term \eqref{error of M} for $k=1, 2$ is estimated in the following way.
Noting $\supp(\hat \phi) \subset [-\a, \a]$
and \eqref{esti p log p}, the error term \eqref{error of M} for $k=1, 2$ is bounded by
\begin{align*}\frac{q^{-1/2}}{(\log q)^2}\sum_{p\le q^{\a/k}}p^{k/2}\log p\ll \frac{q^{-1/2}}{(\log q)^2}\times q^{(\a/k)(k/2+1)} \le \frac{q^{3\a/2-1/2}}{(\log q)^2}.
\end{align*}
Therefore it suffices to take $\a$ so that $3\a/2-1/2\le 0$, i.e., $\a\le 1/3$
and we are done.
\end{proof}

A direct computation shows that the density function of the main term of Proposition \ref{asymptotics L(1/2)} is equal to $W_{\U}^{1}(x)$.
We state this fact as a proposition and provide the proof for convenience of the reader.

\begin{prop}\label{density}
Let $\phi$ be a Schwartz function on $\RR$
such that $\supp(\hat\phi) \subset [-1, 1]$.
Then we have
\begin{align*}
\hat\phi(0)-\phi(0)+\int_{-\infty}^{\infty}\hat\phi(x)|x|dx	
=\int_{-\infty}^{\infty}\phi(x)W_{\rm U}^1(x)dx=
\int_{-\infty}^{\infty}\phi(x)\left(1-\frac{\sin^2(\pi x)}{(\pi x)^2}\right)dx.
\end{align*}
\end{prop}
\begin{proof}Noting the assumption $\supp(\hat\phi) \subset [-1, 1]$,
an elementary calculation using Fourier analysis
yields
\[\hat\phi(0)=\int_{-\infty}^{\infty}\phi(x)dx,\]
\[\phi(0)=\int_{-\infty}^{\infty}\hat\phi(x)dx=\int_{-\infty}^{\infty}\hat\phi(x)\eta(x)dx = \int_{-\infty}^{\infty} \phi(x) \hat{\eta}(x)dx,\]
\[\int_{-\infty}^{\infty}\hat\phi(x)|x|dx=\int_{-\infty}^{\infty}\hat\phi(x)|x|\eta(x)dx = \int_{-\infty}^{\infty} \phi(x) \widehat{|\cdot|\eta}(x)dx,\]
where $\eta(x)=1, 1/2$ or $0$ if $|x|<1$, $|x|=1$ or $|x|>1$, respectively.
We see $\hat\eta(x)=2\frac{\sin (2\pi x)}{2\pi x}$
by a direct computation
and furthermore, we have
\begin{align*}\widehat{|\cdot|\eta}(\xi) = 
\int_{-1}^1|x|e^{2\pi ix\xi}dx
=\int_{-1}^{0}(-x)e^{2\pi i x\xi}dx+\int_{0}^{1}xe^{2\pi i x\xi}dx
= 2\frac{\sin(2\pi\xi)}{2\pi\xi}-\frac{\sin^2(\pi\xi)}{(\pi\xi)^2}.
\end{align*}
Hence, the left-hand side of the assertion is transformed into
\begin{align*}
\int_{-\infty}^{\infty}\phi(x)\left\{1-2\frac{\sin (2\pi x)}{2\pi x}
+\left(2\frac{\sin(2\pi x)}{2\pi x}-\frac{\sin^2(\pi x)}{(\pi x)^2}\right)\right\}dx.
\end{align*}
This completes the proof.
\end{proof}

Theorem \ref{main by L(1/2)} follows from
Propositions \ref{asymptotics L(1/2)} and \ref{density}.

\begin{rem} \label{rmk_supp}
If the error term of Selberg's formula of the twisted second moment of Dirichlet $L$-functions in Theorem \ref{Selberg original} is improved to $\Ocal(m^aq^{-b})$ for some $a>0$ and $b>0$,
then the assumption on the size of $\supp(\hat\phi)$ in Theorem \ref{main by L(1/2)}
can be weakened to $\supp(\hat\phi)\subset [-\frac{b}{a+1/2}, \frac{b}{a+1/2}]$.
On the other hand, the assumption of the size of $\supp(\hat\phi)$ in Theorem \ref{main by L(s)} can be weakenend only up to $\supp(\hat\phi)\subset[-2,2]$ even if
the error term of Selberg's formula is improved to $\Ocal(m^aq^{-b})$ for some $a>0$ and any sufficiently large $b>0$. This is because the main term of Selberg's formula yields the error term $\Ocal_s(m^{s-1}q^{1-2s})$ in Proposition \ref{Selberg by L(s)}, which requires the assumption $\supp(\hat\phi)\subset[-2, 2]$
as in the last part of the proof of Proposition \ref{asymptotics L(s)}.
\end{rem}

\begin{rem}
Selberg's result {\rm(}Theorem \ref{Selberg original}{\rm)} has been improved by many authors for the case of central point, that is when $s=s'=1/2$.
The latest and currently best known result is the following result by Bettin \cite[Corollary 2]{Bettin}{\rm:}
\begin{align*}
\sum_{\chi\in {\Fcal_q}} |L(1/2,\chi)|^2 \chi(m)\overline{\chi(n)}
= \frac{q-1}{(mn)^{1/2}}\left(\log\frac{q}{mn} + \gamma - \log8\pi \right)
+ \Ocal\left((m+n)^{1/2}q^{1/2}\log{q}\right),
\end{align*}
where $q$ is a prime number, $m$ and $n$ are positive integers such that $m$ and $n$ are coprime to each other and such that $q\ge4mn$.
Using the above result, we have
\begin{align*}
&\frac{1}{\sum_{\chi \in \Fcal_q}|L(1/2, \chi)|^2}\sum_{\chi\in {\Fcal_q}}|L(1/2, \chi)|^2 \chi(m) \\
&= m^{-1/2}-m^{-1/2}\frac{\log m}{\log q} +\Ocal\left(\frac{m^{-1/2}}{\log q}\left(1+\frac{\log{m}}{\log q}\right)\right) +\Ocal \left(m^{1/2}q^{-1/2}\right), \quad q \rightarrow \infty.
\end{align*}
This is the case when $a$ and $b$ in Remark \ref{rmk_supp} are $1/2$, and thus the support condition in Theorem \ref{main by L(1/2)} can be improved
from $[-1/3,1/3]$ to $[-1/2,1/2]$.
\end{rem}

\section*{Acknowledgement}

We thank Sandro Bettin for helpful comments and discussions, and for introducing useful references.
This work was supported by JSPS KAKENHI Grant Numbers 20K14298, 18K13400 and 22K13895.
The second author was also supported by MEXT Initiative for Realizing Diversity in the Research Environment.


\end{document}